\documentclass[letterpaper, 10 pt, conference]{ieeeconf}  

\IEEEoverridecommandlockouts                              

\usepackage{ntheorem}
\usepackage{graphics} 
\usepackage{epsfig} 
\usepackage{amsmath} 
\usepackage{amssymb}  

\usepackage{xargs, xcolor}

\newtheorem{assumption}{Assumption}
\newtheorem{example}{Example}
\newtheorem{definition}{Definition}

\newtheorem{proposition}{Proposition}
\newtheorem{remark}{Remark}
\newtheorem{lemma}{Lemma}

\newtheorem{theorem}{Theorem}
\title{\LARGE \bf
Modified Control Barrier Function for Quadratic Program Based Control Design via Sum-of-Squares Programming*
}

\author{Yankai Lin, Michelle S. Chong, and Carlos Murguia
\thanks{*The research leading to these results has received funding from the European Union’s Horizon Europe programme under grant agreement No 101069748 – SELFY project.}
\thanks{Yankai Lin is with the School of Computer Science and Engineering, Wuhan Institute of Technology, PR China. Michelle S. Chong and Carlos Murguia are with Department of Mechanical Engineering, Eindhoven University of Technology, the Netherlands. Carlos Murguia is also with Engineering Systems Design, Singapore University of Technology and Design, Singapore. This work was done while the first author was with the Department of Mechanical Engineering, Eindhoven University of Technology, the Netherlands.
        {\tt\small 25021103@wit.edu.cn;\{m.s.t.chong,c.g.murguia\}@ tue.nl; murguia\_rendon@sutd.edu.sg.}}%
}

\begin{document}

\maketitle

\thispagestyle{empty}
\pagestyle{empty}

\begin{abstract}

We consider a nonlinear control affine system controlled by inputs generated by a quadratic program (QP) induced by a control barrier functions (CBF). Specifically, we slightly modify the condition satisfied by CBFs and study how the modification can positively impact the closed loop behavior of the system. We show that, QP-based controllers designed using the modified CBF condition preserves the desired properties of QP-based controllers using standard CBF conditions. Furthermore, using the generalized S-procedure for polynomial functions, we formulate the design of the modified CBFs as a Sum-Of-Squares (SOS) program, which can be solved efficiently. Via a numerical example, the proposed CBF design is shown to have superior performance over the standard CBF widely used in existing literature.
\end{abstract}

\section{Introduction}

In recent years, the safety of dynamical systems have gained increasing attention by researchers due to its significance in modern engineering applications. Typical examples include autonomous driving, power and electricity networks, and robotics. A commonly used notion of safety for dynamical systems is that the states of the system always remain within a prescribed set, called the \textit{safe set}. A natural tool from dynamical systems theory to study such properties is the invariance of sets \cite{blanchini1999set}. An invariant set has the property that all state trajectories initialized inside it will never leave it in the future. Earlier works attempt to give Lyapunov-like sufficient conditions to certify safety of a set, see \cite{prajna2007framework,tee2009barrier}. However, no control input is considered in \cite{prajna2007framework} and the design proposed in \cite{tee2009barrier} is complicated and might not work well in the presence of input constraints. Other alternatives to guaranteeing the safety of a control system include reachability analysis via Hamilton-Jacobi-Issacs equations \cite{margellos2011hamilton,bansal2017hamilton} and model predictive control (MPC) \cite{rawlings2017model}. Interested readers are referred to the recent paper \cite{wabersich2023data} where these two methods are comprehensively compared and related to the control barrier function (CBF) method, which is the main focus of the paper.

The notion of CBF ensures simultaneously the forward invariance and asymptotic stability of the safe set by imposing a point wise condition on the control input\cite{ames2016control,ames2019control}. These notions are then extended to cover different situations such as robust safety problems \cite{kolathaya2019input,jankovic2018robust}, adaptive safety problems \cite{xiao2021adaptive}, sampled-data systems \cite{breeden2022control}, and so on. A highlight of the notion of CBF is that it can be cast as a constraint of a quadratic program (QP) such that the control law is easy to implement. Moreover, unlike barrier certificate methods \cite{prajna2007framework,wang2023safety,lin2023secondary} where stability is not guaranteed, the QP-based control design can incorporate the notion of control Lyapunov function (CLF) \cite{freeman2008robust} in the same optimization problem. By formulating the CLF constraint as a soft constraint \cite{ames2016control}, the QP will always have a solution. However, the stability of the origin cannot be ensured for the QP-based design proposed in \cite{ames2016control}. Moreover, the QP-based control design might even introduce undesired equilibria \cite{tan2024undesired, reis2020control}. To deal with these issues, modified QPs are proposed. The approaches in \cite{jankovic2018robust} and \cite{tan2024undesired} share similar insights by requiring the existence of an admissible input that satisfies the CLF condition around the origin. While \cite{mestres2023optimization} formulate CLF as a penalty term in the cost function of the QP rather than its constraint.

Although a sufficient amount of effort has been made to investigate the QP part of the design process, much less (if any) attention has been paid to the CBF part of the design. In almost all of the aforementioned works, no detailed analysis can be found on how certain parameters of the CBF function are selected. This issue is touched upon by some works using control barrier certificate (CBC) approach \cite{wang2023safety,schneeberger2023sos,lin2023secondary}. In these works, the synthesis of the control input and barrier certificate is done in a single Sum-Of-Squares (SOS) program. Although it can be observed in simulations that the choice of CBCs indeed affects the performance of the system, there is no analysis of how this is done. Moreover, the entire synthesis is done via an SOS program which generates a polynomial feedback control law that does not have the point-wise minimum norm properties of QP-based design.

In this work, we first modify the conventional definition of CBF proposed in \cite{ames2016control}. And we show that QP-based control design using the modified CBF does preserve all desirable properties of the QP-based design using standard CBFs. Then, we consider systems with polynomial dynamics and formulate the choice of CBF parameters as an SOS program. The thesis by Parillo \cite{parillo2000structured}, shows that an SOS program can be reformulated as a hierarchy of semidefinite programs by using SOS polynomials of increasing degree. Consequently, the selection of the CBF parameter to optimize a given performance index can be efficiently dealt with if we choose to work with polynomial functions.\linebreak

\vspace{-3mm}
\noindent
Our contributions are summarized below.

\noindent
\textbf{1)} We modify the standard definition of CBF to allow more flexibility in selecting CBFs that can potentially lead to better performance of the QP-based controller. The corresponding QP formulated using the modified CBF is shown to inherent all nice properties of the standard QP-based controllers, namely a locally Lipschitz control law, no undesirable equilibria other than the origin, and a locally asymptotically stable origin.

\noindent
\textbf{2)} By working with polynomial functions, we illustrate how the modified CBF that optimize some performance index can be found using SOS programming. To the best of our knowledge, this is the first work that attempt to investigate how CBF parameters can impact the closed loop of QP-based control systems.

The rest of this paper is organized as follows. We first present the preliminaries in Section \ref{pre}, followed by recalling the QP-based design for safety and stability in Section \ref{secQP}. Section \ref{sec-modi-CBF} presents a modified CBF condition whose design will be discussed in Section \ref{sec4}. Section \ref{sec5} illustrates the main results via numerical simulations. Lastly, conclusions and future research directions are given in Section \ref{sec6}.

\section{Preliminaries}\label{pre}
\subsection{Notations}
Let $\mathbb{R}=(-\infty,\infty)$, $\mathbb{R}_{\geq 0}=[0,\infty)$, $\mathbb{R}_{>0}=(0,\infty)$ and $\mathbb{R}^{n}$ denotes the $n$-dimensional Euclidean space. For $x\in\mathbb{R}^n$, $||x||$ denotes its Euclidean norm. A function $f(x)$ is said to be locally Lipschitz continuous at $x$ if there exist $L,\varepsilon\in\mathbb{R}_{>0}$ (that may depend on $x$) such that $||f(a)-f(b)||\leq L||a-b||$ whenever $||a-x||\leq\varepsilon$ and $||b-x||\leq\varepsilon$. Given a differentiable function $V(x)$ and a vector field $f(x)$, the notation $L_fV(x)$ stands for $\frac{\partial V}{\partial x}f(x)$. The interior and boundary of a set $\mathcal{C}\subset \mathbb{R}^n$ are denoted by $\mathrm{Int}\mathcal{C}$ and $\partial C$, respectively. A continuous function $\alpha: [0, a)\rightarrow[0,\infty)$ for $a\in\mathbb{R}_{>0}$ is a class $\mathcal{K}$ function, denoted by $\alpha\in\mathcal{K}$, if it is strictly increasing and $\alpha(0)=0$. A continuous function $\alpha: (-b,a)\rightarrow(-\infty,\infty)$ for $a,b\in\mathbb{R}_{>0}$ is an extended class $\mathcal{K}$ function, denoted by $\alpha\in\mathcal{K}_e$, if it is strictly increasing and $\alpha(0)=0$. For a given square matrix $R$, $\text{Tr}[R]$ denotes the trace of $R$. We use $A\succ 0$ ($A\prec 0$) and $A\succeq 0$ ($A\preceq 0$) to denote the matrix $A$ is positive (negative) definite and positive (negative) semidefinite, respectively. Given a polynomial function $p(x):\mathbb{R}^n\rightarrow\mathbb{R}$, $p$ is called SOS if there exist polynomials $p_i:\mathbb{R}^n\rightarrow\mathbb{R}$ such that $p(x) =\sum_{i=1}^{k}(p_i(x))^2$. The set of SOS polynomials and set of polynomials with real coefficients in $x$ are denoted by $\Sigma[x]$ and $\mathbb{R}[x]$, respectively. A vector of dimension $n$ composed of SOS (real) polynomial functions of $x$ is denoted by $\Sigma^n[x]$ ($\mathbb{R}^n[x]$).

\subsection{Preliminaries on Polynomial Functions}
A standard SOS program is a convex optimization problem of the following form \cite{blekherman2012semidefinite}
\begin{equation}\label{pre-sos}
    \begin{split}
        \underset{m}{\min}&\ b^\top m \text{ such that }\ p_i(x,m)\in\Sigma[x],\ i=1,2,\ldots,n,
        \end{split}
\end{equation}
where $p_i(x,m)=c_{i0}(x)+\sum_{j=1}^{k}c_{ij}(x)m_j$, $c_{ij}(x)\in\mathbb{R}[x]$, and $b$ is a given vector. It is shown in \cite[p. 74]{blekherman2012semidefinite} that (\ref{pre-sos}) is equivalent to a semidefinite program.

A useful tool that will be used extensively in this paper is the generalization of the S-procedure \cite{boyd1994linear} to polynomial functions. This can be done via the Positivstellensatz certificates of set containment \cite{bochnak1987geometrie}. The proof of the following key result can be found in \cite[Chapter 2.2]{jarvis2003lyapunov}.

\begin{lemma}\label{S-pro}
Given $p_0$, $p_1$, $\cdots$, $p_m\in\mathbb{R}[x]$, if there exist $\lambda_1$, $\lambda_2$, $\cdots$, $\lambda_m\in\Sigma[x]$ such that 
    $p_0-\sum_{i=1}^{m}\lambda_ip_i\in\Sigma[x]$, then we have $\overset{m}{\underset{i=1}\bigcap}\{x|p_i(x)\geq0\}\subseteq\{x|p_0(x)\geq0\}$.

\end{lemma}

\subsection{Safety for Dynamical Systems}
We consider the following nonlinear control affine system
\begin{equation}\label{plt}
\begin{split}
    \dot{x}&=f(x)+g(x)u,
    \end{split}
\end{equation}
where $x\in\mathbb{R}^{n}$ is the state vector, and $u\in\mathbb{R}^{m}$ is the input. The functions $f: \mathbb{R}^{n}\rightarrow\mathbb{R}^{n}$ and $g: \mathbb{R}^{n} \rightarrow \mathbb{R}^{n} \times \mathbb{R}^{m}$ are assumed to be locally Lipschitz continuous in $x$ and $f(0)=0$. From standard theory of nonlinear systems \cite{khalil2002nonlinear}, we know that, if $u$ is also Lipschitz in $x$, then for any initial condition $x_0\in\mathbb{R}^n$, there exists a maximal time interval $I(x_0)\subseteq\mathbb{R}$ where \eqref{plt} has a unique solution. A set $\mathcal{S}\subseteq\mathbb{R}^n$ is called \textit{(forward) invariant} with respect to \eqref{plt} if for every $x_0\in\mathcal{S}$, $x(t)\in\mathcal{S}$ for all $t\in I(x_0)$. 

Consider a safe set defined by
\begin{equation}
    \label{def-safe-set}
    \mathcal{S}=\{x\in\mathbb{R}^n: h(x)\geq 0\},    
\end{equation}
for some smooth function $h(x):\mathbb{R}^n\rightarrow \mathbb{R}$. The system \eqref{plt} is called \textit{safe}, if $\mathcal{S}$ is forward invariant. The following is a definition of (zeroing) CBFs \cite{ames2016control} that give sufficient conditions on the input $u$ to ensure safety. 
\begin{definition}\label{def-CBF}
    Consider the safe set $\mathcal{S}$ defined by \eqref{def-safe-set}, $h(x)$ is a control barrier function for system \eqref{plt} if there exists a locally Lipschitz function $\alpha\in\mathcal{K}_e$ such that:
    \begin{equation}
        \label{eq-CBF}
        \underset{u\in\mathbb{R}^m}{\sup}[L_fh(x)+L_gh(x)u+\alpha(h(x))]\geq 0,\ \forall x\in \mathbb{R}^n.
    \end{equation}
\end{definition}
Note that, Definition \ref{def-CBF} requires the CBF condition to hold globally. This allows initial conditions outside the safe set to eventually converge to $\mathcal{S}$ with a properly designed control input. All results stated in this paper remain true if $h(x)$ is defined only over an open subset of $\mathbb{R}^n$ or reciprocal CBF is used (see \cite[Definition 4]{ames2016control}).

In some scenarios, in addition to safety, asymptotic stability of the origin is also required. This motivates the definition of CLF.  
\begin{definition}\label{def-CLF}
    A smooth positive definite function $V:\mathbb{R}^n\rightarrow\mathbb{R}_{\geq 0}$ is a control Lyapunov function (CLF) for \eqref{plt} if there exists a locally Lipschitz $\gamma\in\mathcal{K}$ such that:
    \begin{equation}
        \label{eq-CLF}
        \underset{u\in\mathbb{R}^m}{\inf}[L_fV(x)+L_gV(x)u+\gamma(V(x))]\leq 0,\ \forall x\in \mathbb{R}^n.
    \end{equation}
\end{definition}
The following assumption is made throughout the paper.
\begin{assumption}
    \label{assumption-basic}
    There exist a CLF $V(x)$ and a CBF $h(x)$ for system \eqref{plt}. Furthermore, the origin is in the interior of $\mathcal{S}$, i.e. $0\in\mathrm{Int}\mathcal{S}$.
\end{assumption}

\section{QPs for safety and stability}\label{secQP}
Ideally, we want to design a state feedback control law $u(x)$ that is locally Lipschitz\footnote{A locally Lipschitz feedback control law $u(x)$ ensures that the closed loop system of \eqref{plt} is also locally Lipschitz continuous, such that a unique solution to $\dot{x}=f(x)+g(x)u(x)$ exists for any initial condition $x_0$ within a maximal time interval $I(x_0) \subseteq \mathbb{R}$. } in $x$ such that $\mathcal{S}$ is forward invariant and the origin is globally asymptotically stable. However, this might be impossible for general control affine systems and CBF $h(x)$. In \cite{ames2016control}, a QP-based formulation is proposed which relaxes the CLF condition to ensure feasibility of the QP:
\begin{equation}
    \label{QP-Ames}
    \begin{split}
        &\underset{(u,\delta)\in \mathbb{R}^{m+1}}{\text{min}}\ ||u||^2+p\delta^2\\
         \text{s.t.}\ &L_fh(x)+L_gh(x)u+\alpha(h(x))\geq 0,\qquad (\mathrm{CBF})\\
         \ &L_fV(x)+L_gV(x)u+\gamma(V(x))\leq \delta.\qquad (\mathrm{CLF})\\
    \end{split}
\end{equation}
The slack variable $\delta>0$ makes the CLF condition a soft constraint, which is penalized by the tuning parameter $p>0$. It is shown in \cite{ames2016control} that, the QP \eqref{QP-Ames} is always feasible under Assumption \ref{assumption-basic}. Moreover, under extra assumptions, the resulting $u(x)$ is locally Lipschitz in $x$. However, as pointed out by \cite{tan2024undesired}, the QP \eqref{QP-Ames} might induce some undesired equilibria in $\partial \mathcal{S}$ and $\mathrm{Int}\mathcal{S}$ depending on the parameter $p>0$. Moreover, due to the introduction of the relaxation variable $\delta$, it is shown in \cite{jankovic2018robust} that for no value of $p$ is the controller induced by \eqref{QP-Ames} guaranteed to achieve local asymptotic stability of the origin. To address stability issues of the QP-based controller \eqref{QP-Ames}, a modified QP is proposed in \cite{tan2024undesired}. Suppose a locally Lipschitz nominal controller $u_{\mathrm{nom}}(x)$ exists. We want to design $u'(x)$ such that 
\begin{equation}
    \label{u-total}
    u(x)=u_{\mathrm{nom}}(x)+u'(x)
\end{equation}
is the overall input. The system \eqref{plt} can be equivalently written as
\begin{equation}\label{plt-modi}
\begin{split}
    \dot{x}&=f'(x)+g(x)u',
    \end{split}
\end{equation}
where $f'(x)=f(x)+g(x)u_{\mathrm{nom}}(x)$. The input $u'$ is then generated by the following QP,
\begin{equation}
    \label{QP-Tan}
    \begin{split}
        &\underset{(u',\delta)\in \mathbb{R}^{m+1}}{\text{min}}\ ||u'||^2+p\delta^2\\
         \text{s.t.}\ &L_{f'}h(x)+L_gh(x)u'+\alpha(h(x))\geq 0,\qquad (\mathrm{CBF})\\
         \ &L_{f'}V(x)+L_gV(x)u'+\gamma(V(x))\leq \delta.\qquad (\mathrm{CLF})\\
    \end{split}
\end{equation}
Define $F_h'(x)=L_{f'}h(x)+\alpha(h(x))$ and $F_V'(x)=L_{f'}V(x)+\gamma(h(x))$. Proposition 4 and Theorem 3 of \cite{tan2024undesired} directly lead to the following result.
\begin{proposition}
    \label{prop-tan}
    Consider system \eqref{plt} with a CLF $V$ and a CBF $h$ with its associated safe set $\mathcal{S}$. Assume one of the following conditions holds: (i) $L_gh(x)\neq0$ for all $x$; or (ii) $\{x\in\mathbb{R}^n:F'_h(x)=0, L_gh(x)=0\}$ is empty. Let the nominal control $u_{\mathrm{nom}}$ satisfy the CLF condition \eqref{eq-CLF}, and the control input in \eqref{u-total} is applied to \eqref{plt} with $u'$ being the solution to \eqref{QP-Tan}, then
    \begin{enumerate}
        \item the control input $u$ in \eqref{u-total} is locally Lipschitz continuous;
        \item the set $\mathcal{S}$ is forward invariant;
        \item the origin is the only equilibrium in $\mathrm{Int}\mathcal{S}$;
        \item no equilibrium exists in $\partial\mathcal{S}$ with $L_gh=0$;
        \item the origin is locally asymptotically stable.
    \end{enumerate}
\end{proposition}
In this work, we choose to work with the QP formulation proposed in \cite{tan2024undesired} mainly for the fact that Proposition \ref{prop-tan} successfully addresses many issues existing in the original QP by \cite{ames2016control}. We want to emphasize that, the discussions in the rest of the paper can be applied to other (both zeroing and reciprocal) CBF based QPs.  

\section{A modified CBF condition}\label{sec-modi-CBF}
As shown in \cite{khalil2002nonlinear}, for any positive definite function $W(x):\mathbb{R}^n\rightarrow \mathbb{R}_{\geq0}$, there exist $\gamma_1,\gamma_2\in\mathcal{K}$ such that
\begin{equation*}
    \gamma_1(||x||)\leq W(x) \leq \gamma_2(||x||).
\end{equation*}
As a result, the class $\mathcal{K}$ function in the CLF condition \eqref{eq-CLF} can be replaced by any positive definite function $W(x)$ which is not necessarily an increasing function of $x$ \cite{freeman2008robust}. It is natural to extend the above observation and replace $\alpha\in\mathcal{K}_e$ used in the CBF condition \eqref{eq-CBF} with a function $U(h(x))$ that is not necessarily increasing.

To the best of our knowledge, most existing works including \cite{ames2016control,ames2019control,jankovic2018robust,tan2024undesired,breeden2022control} consider linear functions of $h(x)$ as the extended class $\mathcal{K}$ function in the CBF condition \eqref{eq-CBF}, i.e. $\alpha(h(x))=\lambda h(x)$ for some $\lambda>0$. Notable exceptions include \cite{schneeberger2023sos,wang2023safety,lin2023secondary} where SOS tools are used to find such functions. However, none of the aforementioned works give results on how the use of a nonlinear function of $h(x)$ can improve the design.

\begin{example}
    \label{exp-CBC}
    It is shown in \cite{wang2023safety} that, a sufficient condition to guarantee safety of $\mathcal{S}$ is that
    \begin{equation}\label{eq-CBC}
    \underset{u\in \mathbb{R}^m}{\sup}[L_fh(x)+L_gh(x)u]\geq 0, \forall x\ \mathrm{such\ that}\  h(x)\leq0.
\end{equation}
    The function $h(x)$ associated with $\mathcal{S}$ is known as the control barrier certificate (CBC) for \eqref{plt}. Suppose for system \eqref{plt}, $f(x)$, $g(x)$, $h(x)$, and $u(x)$ are polynomial functions of $x$ with real coefficients. Then, we can apply Lemma \ref{S-pro} to \eqref{eq-CBC} which leads to the following condition,
    \begin{equation}
        \label{eq-CBC-SOS}
        L_fh(x)+L_gh(x)u+\lambda(x)(h(x))\in\Sigma[x],
    \end{equation}
    for some $\lambda(x)\in\Sigma[x]$. Note that the term $\lambda(x)(h(x))$ in general cannot be represented as $\alpha(h(x))$ for some $\alpha\in\mathcal{K}_e$.  
\end{example}
Motivated by the above example, in this work we focus on $U(h(x))$ taking the form $U(h(x))=\lambda(x)h(x)$, where $\lambda(x):\mathbb{R}^n\rightarrow\mathbb{R}_{\geq0}$ is a locally Lipschitz function of $x$ and $\lambda(x)=0$ only if $h(x)=0$. We will see in the subsequent section that this choice leads to improved region of attraction and simplifies the search for a CBF. To this end, the modified CBF condition is given by
\begin{equation}
        \label{eq-CBF-modi}
        \underset{u\in\mathbb{R}^m}{\sup}[L_fh(x)+L_gh(x)u+\lambda(x)h(x)]\geq 0,\ \forall x\in \mathbb{R}^n.
    \end{equation}
Based on \eqref{eq-CBF-modi}, we modify the QP \eqref{QP-Tan} accordingly as follows
\begin{equation}
    \label{QP-Tan-modi}
    \begin{split}
        &\underset{(u',\delta)\in \mathbb{R}^{m+1}}{\text{min}}\ ||u'||^2+p\delta^2\\
         \text{s.t.}\ &L_{f'}h(x)+L_gh(x)u'+\lambda(x)h(x)\geq 0,\qquad (\mathrm{CBF})\\
         \ &L_{f'}V(x)+L_gV(x)u'+\gamma(V(x))\leq \delta.\qquad (\mathrm{CLF})\\
    \end{split}
\end{equation}
It can be seen that if $\lambda(x)=\lambda$, then \eqref{QP-Tan-modi} reduces to \eqref{QP-Tan} with $\alpha(r)=\lambda r$ for $r\in\mathbb{R}$ and we recover the standard choice of the extended class $\mathcal{K}$ function that is used in practice. 

We will conclude this section with our first main result, by showing that the modified CBF condition \eqref{eq-CBF-modi} and the corresponding QP \eqref{QP-Tan-modi} lead to desirable controller properties while guaranteeing safety and stability. First, we write down the solution to the QP \eqref{QP-Tan-modi}, whose detailed derivation can be found in \cite[Theorem 1]{tan2024undesired}. 
\begin{lemma}
    \label{lem-sln-QP}
    The solution to the QP \eqref{QP-Tan-modi} is given by
    \begin{equation}
        \label{sln-u}
        u'(x)=\begin{cases}
      0, & x\in \Omega_{\overline{\mathrm{cbf}}}^{\overline{\mathrm{clf}}}\cup\Omega_{\mathrm{cbf},1}^{\overline{\mathrm{clf}}},\\
      -\frac{F'_\lambda b_1^T}{||b_1||^2}, & x\in \Omega_{\mathrm{cbf},2}^{\overline{\mathrm{clf}}},\\
      -\frac{F'_{V} b_2^T}{1/p+||b_2||^2}, & x\in \Omega_{\overline{\mathrm{cbf}}}^{\mathrm{clf}}\cup\Omega_{\mathrm{cbf},1}^{\mathrm{clf}},\\
      -v_1b_2^T+v_2b_1^T, & x\in \Omega_{\mathrm{cbf},2}^{\mathrm{clf}},\\
    \end{cases}
    \end{equation}
    where $F_\lambda'(x)=L_{f'}h(x)+\lambda(x)h(x)$, $F_V'(x)=L_{f'}V(x)+\gamma(h(x))$, $b_1(x)=L_gh(x)$, $b_2(x)=L_gV(x)$, $\begin{bmatrix} v_1 \\ v_2 \end{bmatrix} = \begin{bmatrix} 1/p+||b_2||^2  & -b_2b_1^T \\
-b_2b_1^T & ||b_1||^2
 \end{bmatrix}^{-1}   \begin{bmatrix} F_V' \\ -F_\lambda' \end{bmatrix}$, and the domain sets are given by
 \begin{equation*}     \Omega_{\overline{\mathrm{cbf}}}^{\overline{\mathrm{clf}}}:=\{x\in\mathbb{R}^n:F'_V<0,F'_\lambda>0\},
 \end{equation*}
 \begin{equation*}     \Omega_{\mathrm{cbf},1}^{\overline{\mathrm{clf}}}:=\{x\in\mathbb{R}^n:F'_V<0,F'_\lambda=0,b_1=0\},
 \end{equation*}
 \begin{equation*}     \Omega_{\mathrm{cbf},2}^{\overline{\mathrm{clf}}}:=\{x\in\mathbb{R}^n:F'_\lambda\leq0,F'_V||b_1||^2<F'_\lambda b_2b_1^T\},
 \end{equation*}
 \begin{equation*}
     \Omega_{\overline{\mathrm{cbf}}}^{\mathrm{clf}}:=\{x\in\mathbb{R}^n:F'_V\geq0,F'_Vb_1b_2^T<F'_\lambda (1/p+||b_2||^2)\},
 \end{equation*}
 \begin{equation*}
     \Omega_{\mathrm{cbf},1}^{\mathrm{clf}}:=\{x\in\mathbb{R}^n:F'_V\geq0,F'_\lambda=0,b_1=0\},
 \end{equation*}
 \begin{equation*}
 \begin{split}
     \Omega_{\mathrm{cbf},2}^{\mathrm{clf}}:=\{x\in&\mathbb{R}^n:F'_V||b_1||^2\geq F'_\lambda b_2b_1^T,\\
     &F'_Vb_1b_2^T\geq F'_\lambda (1/p+||b_2||^2),b_1\neq0\}.
 \end{split} 
 \end{equation*}
\end{lemma}
The notations are adopted from \cite{tan2024undesired} with a bar meaning inactivity of the corresponding constraint. Moreover, cbf,1 refers to the case when the CBF constraint is active and $b_1=0$, while cbf,2 means the CBF constraint is active and $b_1\neq 0$.

The first main result of the paper is given as follows.
\begin{theorem}
    \label{thm-CBF-modi}
    Consider system \eqref{plt} with a CLF $V$ and a continuously differentiable function $h$ with its associated safe set $\mathcal{S}$ that satisfies \eqref{eq-CBF-modi}. Assume one of the following conditions holds: (i) $L_gh(x)\neq0$ for all $x$; or (ii) the set $\{x\in\mathbb{R}^n:F_{\lambda}'(x)=0, L_gh(x)=0\}$ is empty. Let the nominal control $u_{\mathrm{nom}}$ satisfy the CLF condition \eqref{eq-CLF}, and the control input in \eqref{u-total} is applied to \eqref{plt} with $u'$ being the solution to \eqref{QP-Tan-modi}, then
    \begin{enumerate}
        \item the control input $u$ in \eqref{u-total} is locally Lipschitz continuous;
        \item the set $\mathcal{S}$ is forward invariant;
        \item the origin is the only equilibrium in $\mathrm{Int}\mathcal{S}$;
        \item no equilibrium exists in $\partial\mathcal{S}$ with $L_gh=0$;
        \item the origin is locally asymptotically stable.
    \end{enumerate}
\end{theorem}
\begin{proof}
    1) First, note that $V(x)$ and $h(x)$ are smooth functions of $x$. Moreover, $f(x)$, $g(x)$, $\lambda(x)$, and $\gamma(x)$ are locally Lipschitz. As a result, $F'_\lambda(x)$, $F'_V(x)$, $b_1(x)$, and $b_2(x)$ are all locally Lipschitz in $x$. By \cite[Proposition 4]{tan2024undesired}, $u'(x)$ generated by \eqref{QP-Tan-modi} is locally Lipschitz which means $u(x)=u_{\mathrm{nom}}(x)+u'(x)$ is also locally Lipschitz.
    
    2) By item 1), the closed loop system \eqref{plt-modi} admits a unique solution. Define the following function
    \begin{equation*}
        \label{construct-K}
        \tilde{\alpha}(r)=\mu(r) h(x),
    \end{equation*}
    where 
    \begin{equation*}
        \mu(r)=\underset{\{x\in\mathbb{R}^n:||h(x)||\leq r\}}{\sup} \lambda(x).
    \end{equation*}
    It can be seen that $\tilde{\alpha}(0)=0$ and $\mu(r)$ is non-decreasing. Thus there always exists $\alpha\in\mathcal{K}_e$ such that $\alpha(h(x))\geq \tilde{\alpha}(r)\geq \lambda(x)h(x)$. Consequently, the CBF condition in \eqref{QP-Tan-modi} implies the CBF condition in \eqref{QP-Tan}
    \begin{equation*}
       L_{f'}h(x)+L_gh(x)u'+\alpha(h(x))\geq 0,
    \end{equation*}
    for the constructed $\alpha\in\mathcal{K}_e$. Forward invariance of $\mathcal{S}$ follows from standard application of Nagumo’s Theorem \cite{blanchini1999set}.
    
    3) and 4) By \cite[Theorem 2]{tan2024undesired}, an equilibrium exists on $\partial\mathcal{S}$ if and only if the CBF constraint is active at that point. So if an equilibrium exists in $\mathrm{Int}\mathcal{S}$, then at that point the CBF constraint must be inactive. In this case, the QPs \eqref{QP-Tan} and \eqref{QP-Tan-modi} are equivalent. On the other hand, we have $\lambda(x)h(x)=\alpha(h(x))=0$, for all $x\in\partial\mathcal{S}$. This means the QPs \eqref{QP-Tan} and \eqref{QP-Tan-modi} are also equivalent on $\partial\mathcal{S}$. Therefore, parts 3) and 4) follow from parts 3) and 4) of Proposition \ref{prop-tan}.

    5) Since $u_{\mathrm{nom}}(x)$ satisfies the CLF condition, we have $f'(0)=0$, $F'_V(x)\leq 0$ for all $x\in\mathbb{R}^n$. Furthermore, we have $F'_\lambda(0)=L_{f'}h(0)+\lambda(0)(h(0))=\lambda(0)(h(0))>0$. Then there exists a neighborhood of the origin where $F'_\lambda(0)>0$. In this region, the CBF constraint is inactive. Thus the input generated by \eqref{QP-Tan} and \eqref{QP-Tan-modi} coincide. The claim then follows from item 5) of Proposition \ref{prop-tan}.
\end{proof}

\section{Design of $\lambda(x)$}\label{sec4}
In this section, we illustrate how to select a non-constant $\lambda(x)$ to improve the performance of the system. We will assume functions $f(x)$, $g(x)$, $h(x)$, $V(x)$, $\gamma(x)$, $u_\mathrm{nom}(x)$ are all polynomial functions of $x$ with real coefficients. We will use the generalized S-procedure for polynomial functions in Lemma \ref{S-pro} to certify set containment conditions. These conditions are then formulated as SOS optimization problems.  

\subsection{Region of attraction}
Since the QP-based controller \eqref{QP-Tan-modi} is only locally asymptotically stable, its region of attraction (ROA) \cite{khalil2002nonlinear} is of particular interest. The ROA is the set of all initial conditions from which the solution of the system converges to the equilibrium point $x^*$. Finding the ROA analytically can be challenging and hence estimating the ROA is the focus of this section. We will show that with the modified CBF condition \eqref{eq-CBF-modi} via the introduction of the locally Lipschitz function $\lambda(x)$, we are able to formulate the estimation of the ROA in the form of a programme which can be solved with computational tools. 

Under the assumptions in Theorem \ref{thm-CBF-modi}, we have $F'_V(x)\leq 0$ for all $x\in\mathbb{R}^n$. As a result, from \eqref{sln-u}, we know that $u'(x)=0$ and $u(x)=u_{\mathrm{nom}}(x)$ for  $x\in\Omega_{\overline{\mathrm{cbf}}}^{\overline{\mathrm{clf}}}\cup\Omega_{\mathrm{cbf},1}^{\overline{\mathrm{clf}}}\cup\Omega_{\overline{\mathrm{cbf}}}^{\mathrm{clf}}\cup\Omega_{\mathrm{cbf},1}^{\mathrm{clf}}$. Additionally, when the CLF constraint is inactive, $\delta(x)=0$ and the CLF condition \eqref{eq-CLF} is satisfied. In a nutshell, $L_{f'}V(x)+L_gV(x)u'+\gamma(V(x))\leq 0$ might be violated only if $x\in\Omega_{\mathrm{cbf},2}^{\mathrm{clf}}$. Thus, the sub-level set $\mathcal{R}_\eta=\{x\in\mathbb{R}^n:V(x)\leq \eta\}$ can be used to estimate the ROA of \eqref{plt} for any $\eta$ such that $\mathcal{R}_\eta\cap \Omega_{\mathrm{cbf},2}^{\mathrm{clf}}=\emptyset$. 

The following result formulates the estimation of ROA into a standard SOS program.
\begin{proposition}\label{prop-ROA}
    Suppose all conditions in Theorem \ref{thm-CBF-modi} hold. If there exist $\lambda\in\Sigma[x], \lambda_1\in\Sigma[x],\lambda_2\in\Sigma[x]$, $\varepsilon>0$ and $\eta>0$ such that
    \begin{equation}\label{sos-ROA}
        \begin{split}
    &\underset{\lambda,\lambda_1,\lambda_2,\varepsilon,\eta}{\mathrm{max}} \eta\\ 
        \text{s.t.}\ &\lambda-\varepsilon\in\Sigma[x],\\
            &V-\eta-\lambda_1(F'_V||b_1||^2-F'_\lambda b_2b_1^T)\\
            &-\lambda_2(F'_Vb_1b_2^T-F'_\lambda (1/p+||b_2||^2))\in\Sigma[x],
        \end{split}
    \end{equation}
    then all conclusions of Theorem \ref{thm-CBF-modi} hold for the controller generated by the QP \eqref{QP-Tan-modi}. Moreover, $\mathcal{R}_\eta=\{x\in\mathbb{R}^n:V(x)\leq \eta\}$ is a subset of ROA of \eqref{plt}.
\end{proposition}
\begin{proof}
     Since $\lambda-\varepsilon\in\Sigma[x]$, we have $\lambda(x)>0$ for all $x\in\mathbb{R}^n$. Therefore, the results in Theorem \ref{thm-CBF-modi} apply. Applying the S-procedure for polynomial functions, the second condition in (\ref{sos-ROA}) guarantees that, if $F'_V||b_1||^2-F'_\lambda b_2b_1^T\geq0$ and $F'_Vb_1b_2^T-F'_\lambda (1/p+||b_2||^2)\geq0$, we have $V(x)-\eta\geq0$ which is sufficient to ensure that $\mathcal{R}_\eta\cap \Omega_{\mathrm{cbf},2}^{\mathrm{clf}}=\emptyset$.
\end{proof}

\begin{remark}
    \label{rmk-ROA}
    Given the upper bound on the order of polynomials, the optimization problem \eqref{sos-ROA} is a standard SOS program which can be efficiently solved, by SOSTOOLS \cite{papachristodoulou2013sostools}, for instance. Increasing the order of polynomials might improve the result. Note that, the set $\Omega_{\mathrm{cbf},2}^{\mathrm{clf}}$ in fact excludes $x$ such that $b_1(x)=L_gh(x)=0$. Thus, \eqref{sos-ROA} is actually a relaxed problem allowing $b_1(x)=L_gh(x)=0$. This is needed since we only want to work with semi-algebraic sets which lead to a convex program. The relaxation is exact if $L_gh(x)\neq 0$ for all $x\in\mathbb{R}^{n}$. If $L_gh(x)=0$ for some $x\in\mathbb{R}^{n}$, then by the definition of CBF, we have $F'_\lambda\geq 0$. Moreover, $\{x\in\mathbb{R}^n:F_{\lambda}'(x)=0, L_gh(x)=0\}$ needs to be empty meaning $F'_\lambda>0$ and the CBF constraint is inactive for all $x$ such that $L_gh(x)=0$. By part 5) of Theorem \ref{thm-CBF-modi}, $u_\mathrm{nom}$ is a valid stabilizing control law. However, there is no guarantee that these points belong to the actual ROA of \eqref{plt}.
\end{remark}

\subsection{Search for robust CBFs}
When the dynamics of \eqref{plt} are complicated and of higher orders or subject to disturbances, the term $\lambda(x)$ might be useful in the search of CBFs that will be used in the subsequent QP-based design. We illustrate this point by searching robust CBFs defined in \cite{jankovic2018robust} for systems subject to bounded (not necessarily small) disturbances.

Consider the perturbed system \eqref{plt} for some $w\in\mathbb{R}^w$:
\begin{equation}
    \label{plt-dis}
    \dot{x}=f(x)+g(x)u+p(x)w,
\end{equation}
where $||w||\leq\tilde{w}$. The CBF condition can be modified to account for the disturbance signal $w$,
\begin{equation}
        \label{eq-R-CBF}
        \underset{u\in\mathbb{R}^m}{\sup}[L_fh(x)+L_gh(x)u+\alpha(h(x))]\geq ||L_ph(x)||\tilde{w},
\end{equation}
for all $x\in\mathbb{R}^n$. Moreover, suppose we are given a set
\begin{equation}
    \label{eq-setC}
    \mathcal{C}=\{x\in\mathbb{R}^n:c(x)\geq 0\},
\end{equation}
where $c(x)\in{R}[x]$. We want to find a CBF $h(x)\in{R}[x]$ and its associated safe set such that $\mathcal{S}\subseteq\mathcal{C}$. 

Since $p(x)$ and $w$ might be unknown, instead of directly finding a robust CBF for \eqref{plt-dis}, we find a CBF for \eqref{plt} while maximizing its robusness margin. A necessary and sufficient condition for $h(x)$ to be a valid CBF for \eqref{plt} is $L_fh(x)+\lambda(x)h(x)\geq 0$ if $L_gh(x)=0$. Applying the generalized S-procedure, the above condition holds if there exist $\eta\geq0$ and $\lambda_1(x)\in\mathbb{R}[x]$ such that $L_fh(x)+\lambda(x)h(x)-\eta+\lambda_1(x)L_gh(x)\in\Sigma[x]$. The positive value of $\eta$ can be used as a measure of robustness. A larger $\eta$ means the safety of the unperturbed system \eqref{plt} can be maintained under larger disturbances. Similarly the existence of $\lambda_2(x)\in\Sigma[x]$ such that $c(x)-\lambda_2(x)h(x)$ ensures that $\mathcal{S}\subseteq\mathcal{C}$. We summarize the above discussion in the following proposition. 
\begin{proposition}
    \label{prop-CBF-search}
    Consider system \eqref{plt}, if there exist $\lambda\in\Sigma[x], \lambda_1\in\mathbb{R}[x], \lambda_2(x)\in\Sigma[x],h\in\mathbb{R}[x]$, $\varepsilon>0$, and $\eta\geq0$ such that the following optimization problem is feasible,
    \begin{equation}\label{sos-CBF}
        \begin{split}
    &\underset{\lambda,\lambda_1,\lambda_2,h,\varepsilon,\eta}{\mathrm{max}} \eta\\ 
        \text{s.t.}\ &\lambda(x)-\varepsilon\in\Sigma[x],\\      &L_fh(x)+\lambda(x)h(x)-\eta+\lambda_1(x)L_gh(x)\in\Sigma[x],\\
        & c(x)-\lambda_2(x)h(x)\in\Sigma[x].
        \end{split}
    \end{equation}
    Then $h(x)$ is a valid CBF for \eqref{plt}.
\end{proposition}

\begin{remark}
    \label{rmk-CBF-search}
    The optimization problem \eqref{sos-CBF} contains product terms of polynomial function variables $(\lambda,h)$, $(\lambda_1,\frac{\partial h}{\partial x})$, and $(\lambda_2,h)$. Therefore, the optimization problem \eqref{sos-CBF} is in general non-convex, even when the order of polynomial functions are fixed. However, if $h$ is given, and an upper bound of orders of polynomials are specified, then \eqref{sos-CBF} becomes a standard convex SOS program of the remaining variables. The same argument applies if $\lambda,\lambda_1$, and $\lambda_2$ are fixed. In practice, SOS programs like \eqref{sos-CBF} can be solved in an alternating fashion between variables $(\lambda,\lambda_1,\lambda_2)$ and $h$ with the help of some potential slack variables, see \cite{jarvis2003some,yin2021backward,schweidel2022safe,lin2023secondary,schneeberger2023sos,wang2023safety} for detailed steps of implementations. 
\end{remark}

\begin{remark}
    \label{rmk-coma}
    Another observation is that, suppose a linear $\mathcal{K}_e$ function of $h(x)$ leads to a valid CBF with $\lambda(x)=\lambda>0$. Then, for $x\in \mathrm{Int} S$, we have $h(x)>0$. Hence, the first two conditions of \eqref{sos-CBF} ensure that there always exists another $\lambda(x)\geq \lambda$ such that \eqref{sos-CBF} remains feasible with other variables unchanged (noting that the last condition in \eqref{sos-CBF} does not depend on $\lambda$). As a result, the found CBF with $\lambda(x)=\lambda>0$ together with other variables can be used as the new initial condition to run an alternating algorithm to seek an improved $\eta$. In this sense, if we restrict our attention to the interior of the safe set (where reciprocal CBF can also be used), the choice of $\lambda(x)h(x)$ in the CBF condition can do no worse than its counterpart $\lambda h(x)$.   
\end{remark}

\subsection{Discussions}
There are much more to investigate regarding how different choices of $\lambda(x)$ can impact the closed loop performance. One issue that stands out is the input constraint. From the solution to the QP given in Lemma \ref{lem-sln-QP}, it can be seen from the expression of $F'_\lambda(x)$ that a large $\lambda(x)$ leads to less modification of $u_\mathrm{nom}$. However, it also results in a more aggressive control input which makes the system suffer from potential saturation issues and high sensitivity to measurement noises. There are SOS-based results that try to address this issue from different perspectives \cite{wang2023safety,schneeberger2023sos}. However, there is no detailed discussion on how the choice of $\lambda(x)$ impacts satisfaction of input constraints. 

Another interesting feature of $\lambda(x)$ compared to a naive choice of constant $\lambda$ is that we can assign weights on the safety violation term $\lambda(x)h(x)$ in different locations of the state space. Just like loop-shaping synthesis in the frequency domain. For example, for a second order system, $\lambda(x)=x_1^2+kx_2^2+1$ with a very large positive $k$ means that when $x_2$ is away from the origin, it gains more attention than $x_1$. A detailed research on these topics will be left for future research.

\section{Numerical simulation}\label{sec5}
In this section, we illustrate our main results via numerical simulations on Example 2 of \cite{tan2024undesired} to illustrate the improved robustness margin in the search for a robust CBF which is achieved by solely replacing $\alpha(h(x))$, $\alpha\in\mathcal{K}_e$, with $\lambda(x)h(x)$ for some strictly positive SOS function $\lambda(x)$. 

Consider the following linear control system
\begin{equation}\label{example1}
    \begin{split}
        \dot{x}_1&=-x_2\\
        \dot{x}_2&=-x_1+u.
    \end{split}
\end{equation}
In \cite{tan2024undesired}, it is shown that $h(x)=-0.1x_1^2-0.15x_1x_2-0.1x_2^2+4.9$ is indeed a valid CBF according to \eqref{eq-CBF}, where $\alpha(x)\in\mathcal{K}_e$ is simply $\alpha(x)=x$. In this setup, it is shown that, with $L_gh(x)=0$, $L_fh(x)+\alpha(h(x))=0.0389x_2^2+4.9\geq 0$.

We aim to increase the robust CBF margin $4.9$ by considering the same CBF but with $\alpha(h(x))$ replaced by $\lambda(x)h(x)$, where $\lambda(x)$ is to be designed. Since the set $\mathcal{C}$ is not considered, the last constraint of \eqref{sos-CBF} can be removed. In addition, $h(x)$ is given, we choose $\varepsilon=0.001$ so that $\lambda(x)$ is strictly positive, and require all polynomial functions have orders no larger than 4. Consequently, the simplified optimization problem in Proposition \ref{prop-CBF-search}  is a standard SOS program in $\lambda(x)\in\Sigma[x],\lambda_1(x)\in\mathbb{R}[x]$, and $\eta>0$. We solve the problem using SOSTOOLS \cite{papachristodoulou2013sostools} with SeDuMi being the solver \cite{doi:10.1080/10556789908805766}. It turns out that, by choosing 
\begin{equation*}
\begin{split}
        \lambda(x)&=8.3192x_1^2+22.193x_1x_2+14.7935x_2^2+5.591,\\
        \lambda_1(x)&=-46259237.6433x_1-61679009.3186x_2.
\end{split}
\end{equation*}
the value of $\eta$ can be increased to $9.9$. To further improve the result, the orders of polynomials can be increased. Moreover, as discussed in Remark \ref{rmk-CBF-search}, $h(x)$ can be treated as a design variable and an alternating iteration can be performed.

\addtolength{\textheight}{-12cm}   


\section{Conclusions}\label{sec6}
In this work, a modified CBF condition is proposed. Namely, we replace the class $\mathcal{K}_e$ function of the CBF $h$ with the product $\lambda(x)h(x)$ for some $\lambda(x)>0$. It is shown that QP-based controllers using modified CBF preserves nice properties of standard QP-based controllers. For polynomial systems, we formulate the design of such $\lambda(x)$ as SOS programs which leads to a better result in the numerical simulation compared to an example from recent literature where a standard CBF is used. 

This paper serves as a starting point for investigating how the choice of the CBF and its parameters can impact the resulting QP-based control, its properties and the stability, safety and performance of the closed-loop system.




%

\bibliographystyle{ieeetr}
\bibliography{ref.bib}
\end{document}